\documentclass{amsart}

\usepackage{amssymb}
\usepackage{amsmath}
\usepackage{amsthm}
\usepackage{amsbsy}
\usepackage{bm}

\theoremstyle{plain}
\newtheorem{theorem}{Theorem}[section]
\newtheorem{lemma}[theorem]{Lemma}
\newtheorem{corollary}[theorem]{Corollary}
\newtheorem{proposition}[theorem]{Proposition}

\theoremstyle{definition}
\newtheorem{definition}[theorem]{Definition}

\theoremstyle{remark}
\newtheorem{remark}[theorem]{Remark}
\newtheorem*{acknowledgements}{Acknowledgements}

\newcommand{\M}{\mathcal{M}}

\newcommand{\del}{\partial}
\newcommand{\cc}{\mathcal{C}}
\newcommand{\tc}{\widetilde{\mathcal{C}}}

\begin{document} 

\title[Triangulating a compactified moduli space]{A triangulation of a homotopy-Deligne-Mumford compactification of the Moduli of curves}

\author{Siddhartha Gadgil}

\address{	Department of Mathematics,\\
		Indian Institute of Science,\\
		Bangalore 560012, India}

\email{gadgil@math.iisc.ernet.in}

\date{\today}

\begin{abstract}
We construct a triangulation of a compactification of the Moduli space of a surface with at least one puncture that is closely related to the Deligne-Mumford compactification. Specifically, there is a surjective map from the compactification we construct to the Deligne-Mumford compactification so that the inverse image of each point is contractible. In particular our compactification is homotopy equivalent to the Deligne-Mumford compactification. 
\end{abstract}

\maketitle

\section{Arc systems and Cells}

We construct a compactification of the Moduli space of a surface with at least one puncture which is homotopy equivalent to the Deligne-Mumford compactification of moduli space. Thus, the (co)homology of the Deligne-Mumford compactification can be computed using our traingulated compactification.

It is easy to see that our methods extend to the case of surfaces with boundaries and punctures (including boundary-punctures). We were motivated by an attempt to obtain a combinatorial description of the Heegaard Floer theory of Ozsvath and Szabo~\cite{OZ1}\cite{OZ2}. This can be viewed as defined in terms of counting pairs of maps~\cite{Li}. One expects that we can thus reduce to computing cup products in spaces related to Deligne-Mumford compactifications of Moduli spaces.   

Our starting point is the well known model of the Moduli space in terms of arc systems due to Harer~\cite{Ha}. We begin by recalling this model. We shall use the version of this from hyperbolic geometry due to Bowditch and Epstein~\cite{BE} (see also~\cite{Pe}).

Let $F$ be a fixed punctured surface of finite type. We shall also regard $F$ as a surface with boundary in the natural way. An \emph{arc system} $\alpha$ is a collection of disjoint, essential, pairwise non-isotopic arcs. Throughout we shall regard isotopic arcs as equal. By splitting a surface $F$ along a proper codimension-one manifold $\alpha$ we mean taking the completion of $F-\alpha$ with respect to the restriction of a complete Riemannian metric on $F$.

\begin{definition}\label{proper}
An arc system $\alpha$ is said to be proper if the each component of $F$ split along $\alpha$ is either a disc or an annulus $A$ with exactly one boundary component contained in $\del F$.
\end{definition}

A \emph{weighted} arc system is an arc system $\alpha$ together with positive weights associated to the arcs in $\alpha$, with weighted arc systems with proportional weights regarded as equal. We call $\alpha$ the support of the weighted arc system.

A model $X(F)$ for the product $\M(F)\times\Delta$ of the moduli space of $F$ with the simplex $\Delta$ with vertices boundary components of $F$ is given by weighted arc systems up to homeomorphism with support a proper arc system. This is naturally a subset of the simplicial complex formed by all weighted arc systems, and inherits the topology of this complex. We get a model of $\M(F)$ by fixing a boundary component $B$ of $F$ and considering weighted arc system with proper support disjoint from all other boundary components.

\subsection{Size relations}
 
Observe that for the boundary components in $\del F$ that are contained in an annular component of  $F$ split along $\alpha$ are exactly those that are disjoint from $\alpha$. We shall call such boundary components \emph{small} with respect to $\alpha$ and other boundary components \emph{large}. 

\begin{remark}\label{smallsep}
If $\alpha$ is a proper arc system, then two small boundary components are in different components of $F$ split along $\alpha$.
\end{remark}

\begin{remark}\label{properscc}
A curve system $\alpha$ is proper if and only if every essential simple closed curve in $F$ that is disjoint from $\alpha$ is isotopic to a boundary component, which is necessarily small.
\end{remark}

We describe analogous properties in the the compactification in terms of a \emph{size relation}.

\begin{definition}
A \emph{size relation} on a (finite) set $S$ is an equivalence relation $\sim$  on $S$ together with a partial order $\ll$ on the set of equivalence classes.
\end{definition}

We shall also regard $\ll$ as a partial order on the set $S$ which is compatible with the given equivalence relation.

\begin{definition}
An element $s\in S$ is said to be \emph{small} if there is an element $s'\in S$ such that $s\ll s'$. Otherwise we say $s$ is \emph{large}. 
\end{definition}

Observe that if $S$ is finite then there are large elements. Consider a size relation $\ll$  on the set of boundary components of $F$.

\begin{definition}
A proper arc system $\alpha$ is said to be compatible with $\ll$ if 
\begin{itemize}
\item All boundary components of $F$ that intersect $\alpha$ are equivalent.
\item If a boundary component $C$ of $F$ is disjoint from $\alpha$ and another component $C'$ intersects $\alpha$, then $C\ll C'$. 
\end{itemize}
\end{definition}

Note that any proper arc system $\alpha$ on $F$ is compatible with some size relation. Namely, we define $C\sim C'$ if and only if both $C$ and $C'$ intersect $\alpha$, and $C\ll C'$ if and only if $C$ is disjoint from $\alpha$ and $C'$ intersects $\alpha$. We call this size relation the \emph{minimal} size relation for $\alpha$.

\subsection{Cells in the strata}

A proper arc system $\alpha$ in $F$ corresponds to a cell with dimension $\vert\alpha\vert -1$, with $\vert\alpha\vert$ the number of arcs in $\alpha$ (the dimension is reduced by one due to projectivisation. We shall consider additional cells in various strata of our compactification. As with the Deligne-Mumford compactification, these strata correspond to \emph{curve systems} $\cc$, i.e., collections of disjoint, essential, pairwise non-isotopic curves. 

Choose and fix a size relation ($\sim,\ll)$ on the components of $\del F$. Let $\cc$ be a curve system and let $\tc=\cc\cup\del F$. Consider a size relation $(\sim,\ll)$ on the components of $\tc$ (which we simply call a size relation on $\tc$) extending the given relation on $\del F$. 

We say that two components $C,C'\subset \tc$ are \emph{adjacent} if they are  both contained in the closure of some component of $F-\cc$. Equivalently, they are both contained in the boundary of a component of $F_{\cc}$.

\begin{definition}\label{permit}
The size relation $(\sim,\ll)$ on $\tc$ is said to be permissible if for every $C\subset \cc$,  there is a component $C'\subset\tc$ adjacent to $C$ such that $C\ll C'$. 
\end{definition}

Fix a permissible relation $\ll$ on $\tc$. Let $F_{\cc}$ be the components of the surface obtained by splitting $F$ along $\cc$. For each component $G$ of $F_{\cc}$, the size relation $\ll$ restricts to a size relation on the boundary components of $G$.

\begin{definition}
A collection of arc systems $\gamma(G)$, $G$ a component of $F$ split along $\cc$, is said to be proper if, for each component $G$, $\gamma(G)$ is a proper arc system which is compatible with the restriction of the size relation $(\sim,\ll)$ to $\del G$.
\end{definition}

Observe that if $\cc$ is empty this is just the set of proper arc systems on $F$ compatible with $(\sim,\ll)$. In general, we associate a cell to a curve system $\cc$ and a collection of arc systems $\gamma(\cdot)$, so that the collection $\gamma(\cdot)$ is proper with respect to $\cc$ and some size relation $\ll$. The points in the cell correspond to weighted arc systems in components $G$ with support $\gamma(G)$ considered up to scaling (separately in each component $G$). We shall call this space, which is a product of simplices, $\Delta(\gamma(\cdot),\cc)$. We denote the union of the products of simplices in the strata corresponding to $\cc$ by $X(\cc)$.

Note that if we are given a proper arc system $\gamma(G)$ on each component, we can consider the minimal size relation on the components of  $\del G$ from $\gamma(G)$. We can extend the relations $\sim$ and $\ll$ to $\tc$ by requiring transitivity. However, this may give the relation $a\ll a$, and thus we do not get a partial order on $\tc$. It is easy to see that a collection of arc systems is compatible with respect to some partial order if, for the transitive relations $\sim$ and $\ll$ generated by the minimal size relations, we do not have a relation of the form $a\ll a$. 

Thus, $X(\cc)$ is the subset of the product of the spaces $X(G)$ given by the condition that the minima size relations on each component generate a size relation on $\tc$. The space $X(\cc)$ inherits its topology from this product of spaces.

\section{Gluing arc systems}

Consider a curve system $\cc$ together with a permissible size relation $\ll$ nd let $F_{\cc}$ be as before. We shall relate proper arc systems on $F$ with proper arc systems on the components $G$ of $F_{\cc}$. We say that an arc $\gamma^j\subset \gamma(G)$ in $G$ is infinitesimal if exactly one of its boundary points lies on a small component of $\del G$.

Consider an arc system $\alpha$ on $F$. Assume that $\alpha$ intersects $\cc$ minimally and transversally. It is well known that in this case $\alpha\cap G$ is well defined up to isotopy for each component $G$.

\begin{definition}
The \emph{restriction} of $\alpha$ to a component $G$ is the arc system $\gamma(G)=res_G(\alpha)$ consisting of the arcs in the completion of $G\cap \alpha$ with both end points on large components, with isotopic arcs identified.
\end{definition}

\begin{lemma}\label{restproper}
Suppose $\alpha$ is an arc system whose restriction to each component $F$ is proper and so that $\alpha$ intersects each component of $\cc$. Then $\alpha$ is proper.
\end{lemma}
\begin{proof}
By Remark~\ref{properscc}, it suffices to show that an essential curve $\eta$ that is disjoint from $\alpha$ is homotopic to a small boundary component of $\del F$. Let $\eta$ be an essential simple closed curve in $F$. Assume $\gamma$ intersects $\cc$ minimally.

We first show that $\eta$ is disjoint from $\cc$. Suppose not, of the components of $\cc$ that $\eta$ intersects, let $C$ be a component that is maximal with respect to the given size relation. By Definition~\ref{permit}, there is a component $C'$ adjacent to $C$ with $C\ll C'$. Let $G$ be the component containing $C$ and $C'$ and let $\eta'$ be a component of (the commpletion of) $\eta\cap G$.

By maximality of $C$, the other endpoint of $\eta'$ is also contained in a component $C''$ of $\cc$ so that $C''\ll C'$. Hence both $C$ and $C''$ are small, which contradicts Remark~\ref{smallsep} as the restriction $\gamma(G)$ of $\alpha$ to $G$ is proper.

It follows that $\eta$ is isotopic to a small boundary component $C$ of $\del G$. As $\alpha$ intersects each  component of $\cc$, $C$ is in contained in $
del F$ and is a small boundary component of $F$.

\end{proof}

We next see that any collection of proper arc systems $\gamma(G)$ on components $G$ is the restriction of a proper arc system $\alpha$. Furthermore, the extension of an arc is, in an appropriate sense, at least as large as the given arc. 

\begin{lemma}\label{glue}
Let $\gamma(G)$, $G$ a component of $F_{\cc}$, be a collection of proper arc systems. Then there is a proper arc system $\alpha$ on $F$ so that the following hold.
\begin{enumerate}
\item The restriction of $\alpha$ to $G$ is $\gamma(G)$. 
\item\label{uniqsub} Each component of $\alpha$ intersects a unique component $G_\alpha$ in an arc $\gamma_\alpha$ in $\gamma(G_\alpha)$ and intersects all other components in infinitesimal arcs. 
\item\label{minendp} If the end points of $\gamma_\alpha$ lie on components $C_1$ and $C_{-1}$ of $\tc$, and $C'$ is another component of $\tc$ that $\alpha$ intersects, then $C_1\sim C_{-1}$ and $C_1\ll C'$. 
\item\label{uniqsup} Each component of $\alpha_i$ is contained in a unique component of $\alpha$.   
\end{enumerate}
\end{lemma}
\begin{proof}
We shall extend each arc $\gamma^j\subset\gamma(G)$ to a proper arc $\alpha^k$ in $F$ by attaching infinitesimal arcs disjoint from all the other arcs of the collections $\gamma(G')\subset G'$. By iterating this procedure, we obtain the system $\alpha$.

Let $C_1$ and $C_{-1}$ be the curves in $\del G\subset \cc$ on which the end points of $\gamma^j$ lie. By definition of a proper system, both these curves are large in $G$, in particular $C_1\sim C_{1-}$. Hence, as $\ll$  is permissible, if  $C_1$ is not in $\del F$, then $C_1$ is small in the other component $G_1$ in which it is contained. It follows that it is the boundary of an annulus $A$ in the surface obtained from $G_1$ by splitting along $\gamma(G_1)$. Let $C_2$ be a component of $\del G_1$ that intersects $A$. Extend $\gamma^j$ by an infinitesimal arc from $C_1$ to $C_2$, which can be assumed to be disjoint from any given collection of infinitesimal arcs if needed (choosing $C_2$ appropriately). We temporarily denote this extension $\alpha^k$.

Observe that $C_2$ is large in $G_1$ and hence $C_1\ll C_2$. Thus, by permissibility of $\ll$, if $C_2$ is not in $\del F$, $C_2$ is small in the other component in which it is contained. Iterating the above construction, we get a sequence of components $C_1\ll C_2\ll C_3\ll\dots $ and extensions of $\alpha^k$ of $\gamma^j$ by infinitesimal arcs. As $\tc$ has finitely many components, this process must terminate with some $C_j\subset \del F$ and an extension of $\gamma^j$ to an arc $\alpha^k$ with an end point in $\del F$. The same procedure applied to $C_{-1}$ gives an extension of $\gamma^j$ to a proper arc $\alpha^j$ in $F$.

Applying this procedure to each arc in each arc system $\gamma(G)$ in succession, and noting that this can be done keeping the new arcs disjoint, we get an arc system $\alpha$ whose restriction to each component $G$ is $\gamma(G)$. This is proper by Lemma~\ref{restproper}.

The rest of the statements are evident from the construction.

\end{proof}

The condition~(\ref{uniqsup}) in Lemma~\ref{glue} is purely for notational convenience later. On the other hand, the above proof shows that conditions~(\ref{uniqsub}) and~(\ref{minendp}) are automatically satisfied if $\alpha$ restricts to proper curve systems on each component $G$.

\section{Inclusion maps}

We now describe when one cell is contained in the closure of the other and the associated topology. First, we recall the case of cells when $\cc$ is empty.

In this case, simplices are associated to proper arc systems. Consider two proper arc systems $\alpha$ and $\alpha'$ and the associated cells $\Delta(\alpha)$ and $\Delta(\alpha')$. Then $D(\alpha)$ is contained in the closure of $\Delta(\alpha')$ if and only if $\alpha\subset\alpha'$. A weighed arc system in $\Delta(\alpha)$ can be regarded as a weighted arc system corresponding to $\alpha'$ with weights $0$ for the curves in $\alpha'-\alpha$. This gives a natural topology on $\Delta(\alpha)\cup \Delta(\alpha')$.

Next, we consider inclusions of a cell $\Delta(\gamma(\cdot),\cc)$ in the stratum corresponding to $\cc$ in a cell $\Delta(\beta)$ corresponding to a proper arc system $\beta$ in $F$. Assume $\beta$ intersects $\cc$ minimally. For each component $G$ of $F$ split along $\cc$, let $\beta(G)=\beta\cap G$. The cell $\Delta(\gamma(\cdot),\cc)$ is in the closure of $\Delta(\beta)$ if $\gamma(G)\subset \beta(G)$ for all components $G$. Note that, for each component $G$, as $\gamma(G)$ is proper,  $\beta(G)$ is a union of infinitesimal arcs and an arc system in $G$ that is proper with respect to the minimal order from $\gamma(\cdot)$.

Note that any weighted arc system $\xi$ with support $\beta$ gives, for each component $G$, a weighted arc system $\zeta(G)=Res_G(\xi)$ on $G$ by associating to an arc $\beta^j(G)\in\beta(G)$ the sum of the coefficients of arcs in $\beta$ whose intersection with $G$ is $\beta^j(G)$. A weighted arc system on $G$ with support $\gamma(G)$ can be regarded as a weighted arc system on $\beta(G)$ with weights $0$ for arcs not in $\beta(G)$. In this manner we obtain a natural topology on $\Delta(\beta)\cup\Delta(\gamma(\cdot),\cc)$. 

Finally, consider two cells $\Delta_1= \Delta(\gamma_1(\cdot),\cc_1)$ and $\Delta_2=\Delta(\gamma_2(\cdot),\cc_1)$. For $\Delta_1$ to be contained in $\Delta_2$, we require that $\cc_1\supset \cc_2$. We can then regard the cell $\Delta_1$ as corresponding to cells in the components $G$ of $F$ split along $\cc_2$. We are thus reduced to the previous case.

\section{Compactness}

We next see that the union $\bar{X}(F)=\cup_{\cc} X(F,\cc)$ of the strata we have constructed is compact. Note that $X(F)=X(F,\phi)$ corresponds to the empty collection. We shall often omit $F$ from the notation if it is clear from the context.

\begin{theorem}\label{compactness}
The union $\bar{X}=\cup_{\cc} X(\cc)$ of simplices in strata over all curve systems (up to homeomorphism) is compact.
\end{theorem}
\begin{proof}
We shall prove this by induction on the complexity of the surface $F$. Here the complexity of a surface is the maximum number of arcs in an arc system on $F$.

First, consider a sequence of points  $\xi_i\in X(F)$. As there are only finitely many arc systems up to homeomorphism, by passing to a subsequence we can assume that these  points have support a fixed arc system $\alpha$. Further, as the weights all lie in $[0,1]$ and have sum $1$, by passing to a subsequence, we can assume that the weights converge to numbers $c^j\in [0,1]$ corresponding to the components $\alpha^j$ of $\alpha$, with the sum of the numbers $c^j$ equal to $1$.

Let $\alpha_+\subset \alpha$ be the (non-empty) set of arcs for which the coefficients have positve limit. Let $F_+$ be a regular neighbourhood of the union of the arcs in $\alpha_+$ and the boundary components of $F$ which intersect $\alpha_+$. We let $\cc$ be the collection of curves consisting of the boundary of components of $\del F-F_+$ that are not annuli and central circles of components that are essential annuli (i.e., annuli that are not parallel to a boundary component of $\del F$). Up to isotopy, $F_+$ is a collection of components of the surface obtained from $F$ by splitting along $\cc$. Let $F_1$ be the union of the components not in $F_+$.

By definition of $F_+$, $\alpha_+\subset F_+$ and the weighted arc systems $\zeta_i$ consisting of arcs in $\alpha_+$ with weights those in $\xi_i$ have a limit which is proper. Thus, for each component in $F_+$ we obtain a limiting weighted arc system. 

The surface $F_1$ has a lower complexity than $F_+$. We consider the restriction $\eta_i$ of $\xi_i$ to $F_1$, i.e., the arc system with support the intersection of the support of $\xi_i$ with $F_1$ and with the coefficient of an arc in $\eta_i$ the sum of the coefficients of arcs in $\xi_i$ that contain the given arc. By induction, on passing to a subsequence (and considering weights up to scaling in each component of $F_1$) we obtain a limit, which in general lies in a stratum corresponding to a curve system $\cc'$.

It is easy to see that the sequence $\xi_i$ converges to a point in the stratum corresponding to $\cc\cup \cc'$, with the size relation extended so that for each component $F'$ of $F_+$, the components of $\del F'\cap \del F_+$ are equivalent and large while the other components are small.

In general, as the number of curve systems is finite we can assume that a sequence of points is contained in a fixed stratum corresponding to a curve system $\cc$. We then apply the above argument to each component of the surface split along $\cc$.

\end{proof}

We shall use inductive arguments as in the above theorem. To do this, it is useful to observe a lemma regarding convergence to the compactification.

Suppose a sequence of weighted arc systems $\xi_i\in X$ with support $\alpha$ (assumed fixed) converges to a point $\bar\xi$ in the stratum corresponding to a curve system $\cc$ (with a corresponding permissible size relation). Assume that the sum of the coefficients of each of  the weighted arc systems $\xi_i$ is $1$.  Let $\alpha_+$ be the set of arcs in $\alpha$ whose coefficients do not converge to $0$ and let $\del_+ F$ be the union of boundary components whose coefficients (i.e., the total coefficients of arcs on them) do not converge to $0$. Assume $\alpha$ intersects $\cc$ minimally. Let $F_+$ be the union of the components of $F$ split along $\cc$ that intersect $\del_+ F$.

\begin{lemma}\label{poscomp}
We have $\alpha_+\subset F_+$ and $\alpha_+$ is a proper arc system in $F_+$.
\end{lemma}  
\begin{proof}
First we claim $\alpha$ is disjoint from $\cc$. Suppose not, then let $C\in \cc$ intersect $\alpha$. Then, as the size relation in permissible, $C$ is small in some component $G$ of $F$ split along $\cc$. However, $C$ intersects an arc in $\alpha_+$ whose coefficients do not converge to $0$, and hence do not do so on projectivisation in $G$ (as the total coefficients of the restriction of $\xi_i$ to $F_j$ is at most $1$, the total coefficient of $\xi_i$). This means that $\bar\xi$ contains  an arc  $\gamma^j$ in $G$ with positive coefficient with an endpoint of $\gamma^j$ on  $C$. This contradicts the assumption that $C$ is small in $F_j$.

Thus, $\alpha_+$ is contained in a union of components of $F$ split along $\cc$. It is clear from the definition that these are exactly the components of $F_+$.
 
Next, if $G$ is a component of $F_+$, by the definition of $F_+$ the total coefficients of the restriction of $\xi_i$ to $G$ do not converge to $0$. By passing to a subsequence we can assume they converge to a positive number. Hence, the coefficients of an arc in $\alpha\cap G$ in the projective limit in $G$ converge to $0$ if and only if they converge to $0$ in $\xi_i$ (without projectivising). Hence the support of the limit in $F_j$ is $\alpha_+\cap G$, which is proper by definition of the thin strata. This completes the proof.

\end{proof}

We next observe that each of the strata $X(\cc)$ is contained in the closure of the moduli space $X$. Thus $\bar{X}$ is genuinely a compactification of $X$.

\begin{proposition}
Every point in $\bar{X}$ is the limit of points in $X$.
\end{proposition}
\begin{proof}
A point $\bar\xi\in\bar X$ corresponds to a curve system $\cc$ and a weighted arc system $\bar{\xi}(G)$ in each component $G$ of $F$ split along $C$, with support a collection of proper arc systems $\gamma(G)$ in $F$. We assume that the coefficients $c(\gamma^j)$ of the arcs in each surface $G$ have sum $1$.

By Lemma~\ref{glue}, we can find an arc system $\alpha$ restricting to  the systems $\gamma(G)$ and a bijective correspondence(by statement~(\ref{uniqsup}) between the arcs $\alpha^k$ in $\alpha$ and the union of arcs $\gamma^j$ in the arc systems $\gamma(G)$ given by $\gamma^j\subset \alpha^k$. Hence the given weighted arc systems give a collection of coefficients $c(\alpha^k)=\gamma^j$ associated to the components of $\alpha$.

Consider the permissible size relation $(\sim,\ll)$ associated to the point $\bar\xi$. We associate integers $\kappa(C)$ to the components of $\tc=\cc\cup\del F$ so that if $C\sim C'$ then $\kappa(C)=\kappa(C')$ and if $C\ll C'$, $\kappa(C)>\kappa(C')$. Each arc $\alpha^k$ contains a unique arc $\gamma^j$ in some component $G$. The  end points of $\gamma^j$ lie in components $C$ and $C'$ that satisfy $C\sim C'$. Hence we can define $k(\alpha^j)=k(C)$. Further, as all large boundary components of a component $G$ are equivalent, we can define $\kappa(G)=\kappa(C)$ for $C$ any large boundary component. We then have $\kappa(\alpha_i^k)=\kappa(G)$. 

By statement~(\ref{minendp}) of Lemma~\ref{glue}, if an arc $\alpha^l$ intersects $G$ in an infinitesimal arc, then $\kappa(\alpha^l)>\kappa(G)$. Consider the sequence $\xi_i$ of points in $X$ corresponding to weighted arc systems with support $\alpha$ and with the coefficient of the arc $\alpha^k$ being $\delta^{-\kappa(\alpha^k)}c(\alpha^k)$. We claim that this converges to the point $\bar\xi$. For, the restriction of $\xi_j$ to a component $G$ is the sum of $\delta^{G}\bar\xi_i$ and terms with coefficients $O(\delta^l)$ with $l>k$. It follows that on projectivisation the restrictions converge to $\bar{\xi_j}$. As this holds for all components $G$ of $F$ split along $\cc$, the claim follows.
 
\end{proof}

\section{The canonical map to the Deligne-Mumford compactification}

Let $\M(F)$ denote the Moduli space of curves and $\bar\M(F)$ the Deligne-Mumford compactification. Then $\bar{\M}(F)=\cup_\cc \M(\cc)(F)$ is a union of strata corresponding to curve systems $\cc$. The stratum corresponding to $\cc$ is the moduli space of the surface obtained from $F$ by splitting along $\cc$.

Thus, the map from proper weighted arc systems of a surface to the moduli space of the surface, applied to $F$ and surface obtained by splitting $F$ along curve systems gives a map $\Phi:\bar{X}(F)\to \bar{\M}(F)$

We define the \emph{weight} of a curve $C$ with respect to a weighted arc system to be the maximum of the coefficients of arcs that intersect $C$ essentially.

\begin{theorem}\label{cont}
The map $\Phi$ is continuous.
\end{theorem}
\begin{proof}
Consider first a sequence of points $\xi_i\in X$ converging to a point $\bar\xi\in\bar X$. As in the proof of Theorem~\ref{compactness}, we can assume that $\xi_i$ have a fixed support $\alpha$ and the coefficients of each arc in $\alpha$ converge to a non-negative real number. 

As $\bar{M}$ is compact, some subsequence of $z_i=\Phi(\xi_i)$ converges to a limit $\bar{z}$. Clearly it suffices to show that for every such convergent subsequence $\bar{z}=\Phi(\bar\xi)$. Hence it suffices to consider the case where the sequence $z_i$ converges to a limit $\bar{z}$.

Assume that $\bar{z}$ lies in the stratum corresponding to $\cc$. Let $\del_+ F$ be the set of components of $\del F$ whose coefficients do not converge to $0$ and $\alpha_+$ be the set of arcs in $\alpha$ whose coefficients do not converge to $0$. Let $F_+$ be a regular neighbourhood of $\alpha_+\cup \del_+ F$. Recall that this can also be described in terms of the limit as in Lemma~\ref{poscomp}.

Let $F^+$ be the union of components of $F$ split along $\cc$ that intersect $\del_+ F$. 

\begin{lemma}
We have $F_+=F^+$.
\end{lemma}
\begin{proof}
We shall use the correspondence between $X$ and $\M$ using hyperbolic geometry due to Bowditch-Epstein. In their construction, the weighted arc system is determined by a spine which is the locus of the points for which the shortest arc joining the point to the union of an appropriate collection of horospheres is not unique. The total coefficient of each boundary component is fixed and the horocycles are determined by these coefficients. The weighted arc system is dual to the spine constructed and the weight of an arc is determined by the length subtended by the side dual to this arc in a horosphere.  

The hyperbolic structure on the surface $F^+$, with the curves of in $\cc$ represented by geodesics, converges to a cusped hyperbolic structure on $F^+$. It is easy to see that the spine, and hence the weighted arc system also converge to those for the limiting hyperbolic structure, with the total boundaries of the cusps assigned as $0$ and those of the curves in $F^+\cap\del F$ the corresponding limits. By comparing total coefficients, it also follows that the coefficients of arcs not contained in $F^+$  vanish.

As the coefficients of arcs not contained in $F^+$ vanish, $\alpha_+\subset F^+$. Further, as the total coefficients in each component of $F^+$ do not vanish, the support of the limiting arc system is $\alpha_+$, which is hence proper. The claim follows. 

\end{proof} 

The rest of the proof follows by induction on the complexity of the surface, using the description of $F_+$ in terms of the limiting curve system for $\xi_i$ from Lemma~\ref{poscomp}. Namely, if $\cc'$ is the curve system corresponding to the limit $\bar{\xi}$, we observe that the components of $F$ split along $\cc'$ intersecting $\del_+ F$ are isotopic to the components of $F$ split along $\cc$ intersecting $\del_+ F$. In each of these components continuity follows from that of the Bowditch-Epstein construction. We then proceed by induction to complete the proof.
\end{proof}

\section{Fibres of the canonical map}

Finally, we see that the fibres of $\Phi$ are contractible.

\begin{theorem}
For a point $y\in\bar{\M}$, the fibre $\Phi^{-1}(y)$ is contractible.
\end{theorem}
\begin{proof}
The point $y$ lies in a stratum corresponding to some curve system $\cc$. We again proceed inductively. 

Firstly, choose and fix an isotopy class of complex structures on $F$ corresponding to the point $y$ in moduli space. Let $Z$ be the set of isotopy classes of weighted arc systems that correspond to the given isotopy class of complex structure on $F$. We first show that $Z$ is contractible by induction on the complexity of the surface. We then deduce that the fibre $\Phi^{-1}(y)$ is contractible.

Let $F_0$  be a component of $F$ split along $\cc$ containing a component of $\del F$ and $F_1$ be the union of the other components of $F$ split along $\del F_0$. Let $\cc_1$ consist of the curves in $\cc$ in the interior of $F_1$. 

The point $y$ corresponds to a pair of points $(y_0,y_1)$, with $y_0$ in the moduli space of $F_0$ and $y_1$ in the stratum of the compactification of the moduli space of $F_1$ corresponding to $\cc_1$.

Let $Z_1$ be the set of collections of weighted arc systems on the components of $F_1$ split along $\cc'$, which correspond to a permissible size relation, that map to $y_1$ under the corresponding canonical map. By the induction hypothesis, $Z_1$ is contractible. 

There is a natural projection from $Z$ to $Z_1$ by restricting arc systems and permissible relations. We show that each fibre of this map is a disc. As $Z_1$ is contractible, this shows (using, for instance, ~\cite{DK} or~\cite{La}), that $Z$ is contractible.

Consider a point  $z_1\in Z_1$ and let the corresponding minimal size relation be $(\sim, \ll)$. Let $\cc_0$ be the set of boundary components of $\del F_0\cap\del F_1$ that are large in $\del F_1$ (hence small in $\del F_0)$. 

Then the fibre of $z_1$ in $Z$ corresponds to collections of non-negative real numbers with sum $1$ associated to the components of $\del F_0$ so that the minimal size relation on $F_1$ extends to one on $F$. This means that all the boundary components in $\cc_0$ are small in $F_0$. Further, if $C\subset \del F_0$ is contained in $\del F_1$ and there is a component $C'\subset \cc_0$ with $C\ll C'$ in the minimal order generated by the arc system on $F_1$, then $C$ must be small in $F_0$. 

Let $\del^+ F_0$ consist of the components of $\del F_0$ that are not in $\cc_0$ and are not smaller than any element of $\cc_0$. Note that $\del^+ F_0$ is non-empty as it contains $\del F_0-\del F_1$. As each component in $\del^+ F_0$ is either not in $\del F_1$ or is small in a component of $F_1$, both the relations $\ll$ and $\sim$ restricted to $\del^+ F_0$ are empty. 

Hence a weighted arc system on $F_0$ and the given weighted arc system on $F_1$ form an admissible arc system with respect to a size relation on $\cc_0$ if and only if  the components of $\cc_0$ are all small. It follows that  the fibre of the projection is the simplex spanned by the components of $\del^+ F_0$, and is hence a disc, as claimed. By the induction hypothesis, it follows that $Z$ is contractible.
 
Finally, note that $\Phi^{-1}(y)$ is the quotient of $Z$ by the group of automorphisms of the finite group. By passing to a barycentric subdivision, $\Phi^{-1}(Y)$ is the quotient of a compact, contractible space by a simplicial action of a finite set. This has trivial fundamental group by a theorem of Armstrong (as each group element has a fixed point by, for example, the Lefschetz fixed point theorem). By a theorem of Oliver, $\Phi^{-1}(Y)$ is also acyclic (and locally contractible). By Whiteheads theorem it follows that $\Phi^{-1}(Y)$ is contractible.

\end{proof}

\begin{corollary}
The space $X(F)$ is homotopy equivalent to the Deligne-Mumford compactification of the moduli space of $F$.
\end{corollary}

\section{Local homoeomorphism type}

A natural question to ask is whether the compactification we construct is in fact homeomorphic to the Deligne-Mumford compactification. We show that this is not so as the space $X(F)$ need not be locally an orbifold.

Let $F_1$ be the compact surface of genus two with two boundary components $C_1$ and $C_2$ and let $F_0$ be the surface of genus $0$ with $3$ boundary components. Let $F$ be the surface obtained by identifying two of the boundary components of $F_0$ with the boundary components of $F_1$.

Consider the stratum of the compactification of the moduli space of $F$ corresponding to the curve system $\{C_1,C_2\}$. A point $y$ in this stratum is determined by points $y_i$ in the moduli spaces of the surfaces $F_i$ for $i=0,1$. The moduli space of the surface $F_0$ is a sinlge point, and $y_0$ must be this point. let $y_1$ be a point in the moduli space of $F_1$ with trivial automorphism group.

Consider a point $z$ in the fibre $\Phi^{-1}(y)$ corresponding to weighted arc systems with the total weights of the curves $C_1$ and $C_2$ in $F_1$ both $1/2$. Such a point is unique as the weights of these curves in $F_0$ must be $0$. 

\begin{proposition}
A neighbourhood of the point $z$ in $X(F)$ is homeomorphic to the cone on an iterated suspension of a $2$-torus.
\end{proposition}
\begin{proof}
A neighbourhood of $y$ in the Deligne-Mumford compactification is the product of a neighbourhood $V$ of $y_1$ in $\M(F_1)$ with discs corresponding to length and twist parameters for the curves $C_1$ and $C_2$. The neighbourhood $V$ can be chosen homeomorphic to a ball.

In the space $X(F)$, a neighbourhood of $F$ is a subset $W$ of the product of $U$ with discs corresponding to the length and twist parameters, with $W$ consisting of weighted arc systems with the ratio $\alpha$ of the weights of the curves $C_1$ and $C_2$ close to $1/2$. The subset of $W$ corresponding to a fixed point in $U$ and a fixed ratio $\alpha$ is thus a cone on the torus corresponding to the twist parameters. 

It follows that $W$ is the cone on an iterated suspension of tori.
\end{proof}

We remark that the existence of a cell-like homotopy equivalence between spaces that are not homeomorphic does not contradict~\cite{Se} as the space $X$ is not a manifold.

\begin{acknowledgements}
I thank Dennis Sullivan for his invaluable comments and suggestions.
\end{acknowledgements}

\end{document}